\documentclass[12pt]{amsart}
\usepackage{fancybox}
\usepackage{ascmac}
\usepackage{delarray}
\usepackage{enumerate}
\usepackage{amsmath,amssymb}
\usepackage{color}
\usepackage{comment}

\usepackage{mathrsfs}
\usepackage{paralist}

\makeatletter
 \@addtoreset{equation}{subsection}
\makeatother

\def\H{\mathcal{H}}
\def\W{\mathcal{W}}
\def\re{\mathbb{R}}

\def\eps{\varepsilon}
\def\pd{\partial}

\def\la{\lambda}

\def\({\left(}
\def\){\right)}
\def\pd{\partial}

\def\BOX{{\setlength{\unitlength}{1pt}\begin{picture}(8,8)
	\put(1,1){\framebox(6,6)}\end{picture}}\ }
\def\qed{\hfill\BOX\vskip1em\par}

\def\intRN{\int_{\re^N}}

\def\diver{{\rm div }}

\numberwithin{equation}{section}
\newtheorem{theorem}{Theorem}[section]

\newtheorem{proposition}[theorem]{Proposition}

\newtheorem{remark}[theorem]{Remark}

\numberwithin{theorem}{section}

\begin{document}

\title[$p$-harmonic transplantation]{Applications of $p$-harmonic transplantation for functional inequalities involving a Finsler norm}
\author[S. Habibi]{Sadaf Habibi$^1$}
\author[F. Takahashi]{Futoshi Takahashi$^2$}

\date{\today}

\setcounter{footnote}{1}
\footnotetext{
Department of Mathematics, Osaka City University,
3-3-138 Sugimoto, Sumiyoshi-ku, Osaka 558-8585, Japan. \\
e-mail:{\tt habibzaisadaf@gmail.com}}

\setcounter{footnote}{2}
\footnotetext{
Department of Mathematics, Osaka City University,
3-3-138 Sugimoto, Sumiyoshi-ku, Osaka 558-8585, Japan. \\
e-mail:{\tt futoshi@sci.osaka-cu.ac.jp}}

\begin{abstract} 
In this paper, we prove several inequalities such as Sobolev, Poincar\'e, logarithmic Sobolev, which involve a general norm with accurate information of extremals,
and are valid for some symmetric functions.
We use Ioku's transformation, which is a special case of $p$-harmonic transplantation, between symmetric functions.

\rm 
\\[0.1cm]
\rm 
\end{abstract}
\maketitle

\medskip

\section{Introduction}
\label{section:Introduction}

The Sobolev inequality
\begin{equation}
\label{Sobolev}
	S_{N,p} \( \int_{\re^N} |u|^{p^*} dx \)^{p/p^*} \le \int_{\re^N} |\nabla u|^p dx
\end{equation}
which holds for every $u \in W^{1,p}(\re^N)$, where $N \ge 2$, $1 \le p < N$, and $p^* = \frac{Np}{N-p}$,
is one of the fundamental tools in analysis.
The best constant $S_{N,p}$ is known as
\begin{equation}
\label{SNp}
	\begin{cases}
	S_{N,p} = \pi^{p/2} N \(\frac{N-p}{p-1}\)^{p-1} \( \frac{\Gamma(\tfrac{N}{p}) \Gamma(1+\tfrac{N}{p'})}{\Gamma(N) \Gamma(1 + \tfrac{N}{2})} \)^{p/N} \quad 1 < p < N, \\
	S_{N,1} = \pi^{p/2} \frac{N}{\( \Gamma(1 + \tfrac{N}{2}) \)^{p/N}} \quad p = 1,
	\end{cases}
\end{equation}
where $p' = \frac{p}{p-1}$, see \cite{Federer-Fleming}, \cite{Mazya}, \cite{Aubin} and \cite{Talenti}.
It is well-known that the best constant for $1 < p < N$ is achieved by a family of functions of the form
\begin{equation}
	U(x) = \( a + b |x|^{\frac{p}{p-1}} \)^{-\frac{N-p}{p}}, \quad a, b > 0
\end{equation}
and its translation $U(x - x_0)$ for $x_0 \in \re^N$.
If we replace $\re^N$ by a domain different from $\re^N$, then still Sobolev inequality with the same best constant holds true for functions in $W^{1,p}_0(\Omega)$,
however, the attainability of the constant is lost.

Recently, Ioku \cite{Ioku} obtained a new Sobolev type inequality for radially symmetric functions on the ball $B_R$ with radius $R > 0$, which admits extremals for the best constant in the inequality.

\vspace{1em}\noindent
{\bf Theorem A (Ioku \cite{Ioku})}

{\it Let $N \ge 2$, $1 < p < N$ and $p^* = \frac{Np}{N-p}$.
Then for any radially symmetric function $v \in W^{1,p}_0(B_R)$, the inequality
\[
	S_{N,p} \( \int_{B_R} \frac{|v(y)|^{p^*}}{\( 1 - (\frac{|y|}{R})^{\frac{N-p}{p-1}} \)^{\frac{p(N-1)}{N-p}}} dy \)^{p/p^*} \le \int_{B_R} |\nabla v(y)|^p dy
\]
holds true.
Here $S_{N,p}$ on the left-hand side is the same constant in \eqref{SNp}.
The equality occurs if and only if $v$ is of the form
\[
	v(y) = \( a + b \( |y|^{\frac{p-N}{p-1}} - R^{\frac{p-N}{p-1}} \)^{\frac{p}{p-N}} \)^{\frac{p-N}{p}} 
\]
for some $a, b >0$.
}

\vspace{1em}
Ioku obtained the above theorem by exploiting a transformation between radially symmetric functions on $B_R$ and on the whole $\re^N$.
Then the inequality in Theorem A is the ''disguised" version of \eqref{Sobolev} under the transformation.
His transformation is a special case of $p$-harmonic transplantation initiated by J. Hersh \cite{Hersh}, see also \cite{Bandle-Brillard-Flucher}. 
This point of view is pursued in \cite{Sano-TF}.
Harmonic transplantation is used to prove the attainability of the supremum of Trudinger-Moser type functional, see \cite{Flucher}, \cite{Csato-Roy(CVPDE)}, \cite{Csato-Roy(CPDE)}, \cite{CNR}.
Other applications for functional inequalities, such as Hardy, or Hardy-Sobolev type, see also \cite{Sano-TF(CVPDE)}, \cite{Sano(arXiv)}, \cite{Sano(NA)}, and the reference therein. 

In this paper, we use the same transformation by Ioku,
but between radially symmetric functions on the whole space (resp. on the ball) and {\it Finsler radially symmetric} functions on the ball (resp. on the whole space).
The resulting inequalities include a general norm (Finsler norm) with the precise information on the extremizers of the best constant involved.
For other inequalities involving Finsler norms, we refer to \cite{AFMTV}, \cite{DP-B-G}, \cite{Mercaldo-Sano-TF}, \cite{RS(Adv.Math)}, \cite{Wang-Xia(JDE)}, \cite{Wang-Xia(Pacific)}, \cite{Zhou-Zhou}, 
and the references therein..

\section{An integral formula for Finsler symmetric functions}
\label{section:Finsler}

Let $H : \re^N \to [0, +\infty)$, $N \ge 2$ be a function such that
$H$ is convex, $H(\xi) \ge 0$, $H(\xi) = 0$ if and only if $\xi = 0$, 
and satisfies 
\begin{equation}
\label{homo_H}
	H(t \xi) = |t| H(\xi), \quad \forall \xi \in \re^N, \, \forall t \in \re.
\end{equation}
By \eqref{homo_H}, $H$ must be even: $H(-\xi) = H(\xi)$ for all $\xi \in \re^N$.
We assume in this paper that $H \in C^1(\re^N \setminus \{ 0 \})$.
We call $H$ a {\it Finsler norm} on $\re^N$.
Since all norms on $\re^N$ are equivalent to each other, we see the existence of positive constants $\alpha$ and $\beta$ such that
\begin{equation*}
	\alpha |\xi| \le H(\xi) \le \beta |\xi|, \quad \xi \in \re^N.
\end{equation*}
The dual norm of $H$ is the function $H^0: \re^N \to [0, +\infty)$ defined by 
\[
	H^0(x) = \sup_{\xi \in \re^N \setminus \{ 0 \}} \frac{\xi \cdot x}{H(\xi)}. 
\]
It is well-known that $H^0$ is also a norm on $\re^N$ and satisfies the inequality
\[
	\frac{1}{\beta} |x| \le H^0(x) \le\frac{1}{\alpha} |x|, \quad \forall x \in \re^N.
\]
The set
\[
	\mathcal{W} = \{ x \in \re^N \ | \ H^0(x) < 1 \}
\]
is called the {\it Wulff ball}, or the $H^0$-unit ball, and we denote $\kappa_N = \mathcal{H}^N(\mathcal{W})$,
where $\mathcal{H}^N$ denotes the $N$-dimensional Hausdorff measure on $\re^N$.
We also denote 
\[
	\mathcal{W}_r = \{ x \in \re^N \ | \ H^0(x) < r \}
\]
for any $r > 0$.

Here we just recall some properties of $H$ and $H^0$.
These will be proven by using the homogeneity property of $H$ and $H^0$,
see \cite{Bellettini-Paolini} Lemma 2.1, and Lemma 2.2.

\begin{proposition}
\label{Prop:identities} 
Let $H$ be a Finsler norm on $\re^N$. 
Then the following properties hold true:
\begin{enumerate}
\item[\rm(1)] $|\nabla_{\xi} H(\xi)| \le C$ for any $\xi \ne 0$.
\item[\rm(2)] $\nabla_{\xi} H(\xi) \cdot \xi = H(\xi)$, $\nabla_x H(x) \cdot x = H(x)$ for any $\xi \ne 0$, $x \ne 0$. 
\item[\rm(3)] $\(\nabla_{\xi} H \)(t \xi) = \frac{t}{|t|} \( \nabla_{\xi} H \)(\xi)$ for any $\xi \ne 0$, $t \ne 0$.
\item[\rm(4)] $H\( \nabla_x H^0(x) \) = 1$. $H^0\( \nabla_{\xi} H(\xi) \) = 1$.
\item[\rm(5)] $H^0(x) \( \nabla_{\xi} H \) \( \nabla_x H^0(x) \) = x$.
\end{enumerate}
\end{proposition}

For a domain $\Omega \subset \re^N$ and a Borel set $E \subset \re^N$, 
the {\it anisotropic $H$-perimeter} of a set $E$ with respect to $\Omega$ is defined as
\[
	P_H(E; \Omega) = \sup \left\{ \int_{E \cap \Omega} \diver \sigma dx \ | \ \sigma \in C_0^{\infty}(\Omega; \re^N), H^0(\sigma(x)) \le 1 \right\}.
\]  
If $E$ is Lipschitz, then it holds $P_H(E; \Omega) = \int_{\Omega \cap \pd^* E} H(\nu) d\mathcal{H}^{N-1}$, 
where $\pd^* E$ denotes the reduced boundary of the set $E$ and $\nu(x)$ is the measure theoretic outer unit normal of $\pd^* E$ (see \cite{Evans-Gariepy}).
For more explanation about the anisotropic perimeter, see \cite{AFTL} and {\cite{Belloni-Ferone-Kawohl}.

Now, we prove that $P_H(\mathcal{W}; \re^N) = N \kappa_N$.
If $H^0$ is Lipschitz, then $\pd \mathcal{W}$ is also Lipschitz and $\pd^* \mathcal{W} = \pd \mathcal{W}$ holds.
In this case, the outer unit normal of $\mathcal{W}$ is given by $\nu = \frac{\nabla H^0}{|\nabla H^0|}$.
Thus
\begin{align*}
	P_H(\mathcal{W}; \re^N) &= \int_{\pd \mathcal{W}}  H(\nu(x)) \, d\H^{N-1}(x) \notag \\ 
	&= \int_{\pd \mathcal{W}}  H(\frac{\nabla H^0}{|\nabla H^0|}) \, d\H^{N-1}(x) \notag \\
	&= \int_{\pd \mathcal{W}}  \frac{1}{|\nabla H^0|} \, d\H^{N-1}(x).
\end{align*}
Here we have used $H(\nabla H^0(x)) = 1$ by Proposition \ref{Prop:identities} and the positive 1-homogeniety of $H$.
Similarly, we have
\begin{equation}
\label{P_H1}
	P_H(\mathcal{W}_r; \re^N) = \int_{\pd \mathcal{W}_r}  \frac{1}{|\nabla H^0|} \, d\H^{N-1}(x)
\end{equation}
for any $r > 0$.
On the other hand, 
by the fact $H^0(x) \equiv 1$ on $\pd \mathcal{W}$, the formula $x \cdot \nabla H^0(x) = H^0(x)$ by Proposition \ref{Prop:identities}, and the divergence theorem,
we have
\begin{align*}
	P_H(\mathcal{W}; \re^N) &= \int_{\pd \mathcal{W}}  \frac{1}{|\nabla H^0|} \, d\H^{N-1}(x) \\
	&= \int_{\pd \mathcal{W}}  \frac{H^0(x)}{|\nabla H^0(x)|} \, d\H^{N-1}(x) \\
	&= \int_{\pd \mathcal{W}}  \frac{x \cdot \nabla H^0(x)}{|\nabla H^0(x)|} \, d\H^{N-1}(x) \\
	&= \int_{\pd \mathcal{W}}  x \cdot \nu \, d\H^{N-1}(x) \\
	&= \int_{\mathcal{W}}  {\rm div} x dx = N \H^N(\mathcal{W}).
\end{align*}
Thus we obtain
\begin{equation}
\label{P_H2}
	P_H(\mathcal{W}; \re^N) = N \H^N(\mathcal{W}) = N \kappa_N.
\end{equation}
Similarly, we have
\begin{equation}
\label{P_H3}
	P_H(\mathcal{W}_r; \re^N) = N \kappa_N r^{N-1} \quad (r > 0).
\end{equation}
The equality \eqref{P_H3} is proved as follows:
Recall $\mathcal{W}_r = \{ H^0(x) < r \}$ for $r > 0$.
Then $y = x/r \in \mathcal{W} = \mathcal{W}_1$ for $x \in \mathcal{W}_r$ by the homogeniety of $H^0$. 
On the other hand, since $\nabla H^0$ is $0$-homogeneous, we have $\nabla H^0(x) = \nabla H^0(y)$.
Thus by \eqref{P_H1},  we have
\begin{align*}
	P_H(\mathcal{W}_r; \re^N) &= \int_{\pd \mathcal{W}_r} \frac{d\H^{N-1}(x)}{|\nabla H^0(x)|} = \int_{y \in \pd \mathcal{W}} \frac{d\H^{N-1}(ry)}{|\nabla H^0(ry)|} \\
	&= r^{N-1} \int_{y \in \pd \mathcal{W}} \frac{d\H^{N-1}(y)}{|\nabla H^0(y)|} = r^{N-1} P_H(\mathcal{W}).
\end{align*}

In the following,
we call a function $g$ of the form $g(x) = h(H^0(x))$ for some $h: \re_+ \to \re$ as $H^0$-symmetric , or {\it Finsler symmetric} function.

The following is a key fact in this paper.
Though the statement is widely known, we prove it here for the sake of completeness.

\begin{proposition} {\rm (Polar formula)}
\label{Prop:Polar}
Let $h: \re_+ \to \re$ be such that $h \circ H^0 \in L^1_{loc}(\re^N)$.
Then it holds that
\begin{align}
\label{Polar}
	\int_{H^0(x) < t} h(H^0(x)) dx &= P_H( \mathcal{W} ; \re^N) \int_0^t h(s) s^{N-1} ds \notag \\
	&= N \kappa_N \int_0^t h(s) s^{N-1} ds.
\end{align}
In particular, if $h \circ H^0 \in L^1(\re^N)$, then
\begin{align*}
	\intRN h(H^0(x)) dx = N \kappa_N \int_0^\infty h(s) s^{N-1} ds
\end{align*}
holds.
\end{proposition}

For the proof of Proposition \ref{Prop:Polar}, we use the coarea formula in the following form.

\begin{theorem}{\rm (Coarea formula)}
\label{Theorem:Coarea}
Let $f: \re^N \to \re$ be Lipschitz and let $g \in L^1(\re^N)$.
Then it holds that 
\begin{equation}
\label{coarea}
	\int_{f(x) < t} g(x) dx = \int_0^t \int_{f(x) = s} \frac{g(x)}{|\nabla f(x)|} d\H^{N-1}(x) ds.
\end{equation}
\end{theorem}
See for example \cite{Evans-Gariepy} \S 3.4.4, or \cite{BZ(book)} \S 13.4.

\vspace{1em}\noindent
{\it Proof of Proposition \ref{Prop:Polar}.}

For $h : \re_+ \to \re$ as above, put $g(x) = h(H^0(x))$, $f(x) = H^0(x)$ in Theorem \ref{Theorem:Coarea}.
Then \eqref{coarea}, \eqref{P_H1}, \eqref{P_H3}, and \eqref{P_H2} yield that
\begin{align*}
	\int_{H^0(x) < t} h(H^0(x)) dx &\stackrel{\eqref{coarea}}{=} \int_0^t \int_{H^0(x) = s} \frac{h(H^0(x))}{|\nabla H^0(x)|} d\H^{N-1}(x) ds \\
	&= \int_0^t h(s) \( \int_{\pd \W_s} \frac{1}{|\nabla H^0(x)|} d\H^{N-1}(x) \) ds \\
	&\stackrel{\eqref{P_H1}}{=} \int_0^t h(s) P_H( \W_s ; \re^N) ds \\
	&\stackrel{\eqref{P_H3}}{=} \int_0^t h(s) P_H(\W ; \re^N ) s^{N-1} ds \\
	&= P_H( \mathcal{W} ; \re^N) \int_0^t h(s) s^{N-1} ds \\
	&\stackrel{\eqref{P_H2}}{=} N \kappa_N \int_0^t h(s) s^{N-1} ds.
\end{align*}
\qed

\section{A transformation between symmetric functions}
\label{section:transformation}

Let $u = u(x)$ be a radially symmetric function,
thus, there exists a function $U$ defined on $[0, +\infty)$ such that $u(x) = U(|x|)$.
Also let $v = v(y)$ be a Finsler radially symmetric function on $\mathcal{W}_R$ of the form $v(y) = V(H^0(y))$ for some $V = V(s)$, $s \in [0, R)$,
where $R > 0$ be any number. 
Let us assume that $u$ and $v$ are related with each other by the transformation
\begin{equation}
\label{Ioku}
\begin{cases}
	&r = |x|, x \in \re^N, \\
	&s = H^0(y), y \in \mathcal{W}_R \subset \re^N, \\
	&r^{\frac{p-N}{p-1}} = s^{\frac{p-N}{p-1}} - R^{\frac{p-N}{p-1}}, \\
	&u(x) = U(r) = V(s) = v(y).
\end{cases}
\end{equation}
Throughout of the paper, $\omega_{N-1}$ denotes the surface measure of the unit sphere ${\mathbb S}^{N-1}$ in $\re^N$.
Under the transformation \eqref{Ioku}, we have the following equivalence.

\begin{proposition}
\label{Prop:Ioku}
Let $u, v$ be as above. Let $F: \re \to \re$ be continuous.
Then we have
\begin{align*}
	&\int_{\re^N} |\nabla u|^p dx = \frac{\omega_{N-1}}{N \kappa_N} \int_{\mathcal{W}_R} H(\nabla v)^p dy, \\
	&\int_{\re^N} F(u(x)) dx = \frac{\omega_{N-1}}{N \kappa_N} \int_{\mathcal{W}_R} \frac{F(v(y))}{\( 1 - \(\frac{H^0(y)}{R}\)^{\frac{N-p}{p-1}} \)^{\frac{p(N-1)}{N-p}}} dy,
\end{align*}
\end{proposition}

\begin{proof}
By \eqref{Ioku}, we see that if $r$ runs from $0$ to $\infty$, then $s$ runs from $0$ to $R$, and vice versa.
Also by differentiating the relation $r^{-\frac{N-p}{p-1}} = s^{-\frac{N-p}{p-1}} - R^{-\frac{N-p}{p-1}}$ with respect to $s$,
we have
\[
	\(\frac{p-N}{p-1}\) r^{\frac{p-N}{p-1}-1} \(\frac{dr}{ds}\) = \(\frac{p-N}{p-1}\) s^{\frac{p-N}{p-1}-1},
\]
which implies
\[
	\frac{dr}{ds} = (r(s))^{\frac{N-1}{p-1}} s^{\frac{1-N}{p-1}}, \quad \frac{ds}{dr} = (s(r))^{\frac{N-1}{p-1}} r^{\frac{1-N}{p-1}}.
\]
Since $U'(r) = V'(s) (\frac{ds}{dr})$, we compute
\begin{align}
\label{E1}
	\int_0^\infty |U'(r)|^p r^{N-1} dr &= \int_0^R |V'(s)|^p \(\frac{ds}{dr}\)^p r(s)^{N-1} \(\frac{dr}{ds}\) ds \notag \\
	&= \int_0^R |V'(s(r))|^p \(\frac{ds}{dr}\)^{p-1} r^{N-1} dr \notag \\
	&= \int_0^R |V'(s)|^p \(\frac{ds}{dr}\)^{p-1} r(s)^{N-1} ds \notag \\
	&= \int_0^R |V'(s)|^p \( s^{\frac{N-1}{p-1}} r(s)^{\frac{1-N}{p-1}} \)^{p-1} r(s)^{N-1} ds \notag \\
	&= \int_0^R |V'(s)|^p s^{N-1} ds.
\end{align}
Now, for $v(y) = V(H^0(y))$, $y \in \mathcal{W}_R$, we compute
\begin{align}
\label{E2}
	&\nabla v(y) = V'(H^0(y)) \nabla H^0(y), \notag \\
	&H\(\nabla v(y)\) = H\(V'(H^0(y)) \nabla H^0(y)\) = |V'(H^0(y)| H(\nabla H^0(y)) = |V'(H^0(y)|,
\end{align}
here we used Proposition \ref{Prop:identities}
Recalling Proposition \ref{Prop:Polar} \eqref{Polar} with $h(s) = |V'(s)|^p$, we have
\begin{align*}
	\int_{\re^N} |\nabla u|^p dx &= \omega_{N-1} \int_0^\infty |U'(r)|^p r^{N-1} dr \\
	&\stackrel{\eqref{E1}}{=} \frac{\omega_{N-1}}{N \kappa_N} N \kappa_N \int_0^R |V'(s)|^p s^{N-1} ds \\
	&\stackrel{\eqref{Polar}}{=} \frac{\omega_{N-1}}{N \kappa_N} \int_{\mathcal{W}_R} |V'(H^0(y))|^p dy \\
	&\stackrel{\eqref{E2}}{=} \frac{\omega_{N-1}}{N \kappa_N} \int_{\mathcal{W}_R} H(\nabla v)^p dy.
\end{align*}
On the other hand, we compute
\begin{align}
\label{E3}
	\int_0^\infty F(U(r)) r^{N-1} dr &= \int_0^R F(V(s)) r(s)^{N-1} \(\frac{dr}{ds}\) ds \notag \\
	&= \int_0^R F(V(s)) r(s)^{N-1} r(s)^{\frac{N-1}{p-1}} s^{\frac{1-N}{p-1}} ds \notag \\
	&= \int_0^R F(V(s)) r(s)^{(N-1)(1 + \frac{1}{p-1})} s^{\frac{1-N}{p-1} + 1 -N} s^{N-1} ds \notag \\
	&= \int_0^R F(V(s)) \( \( s^{\frac{p-N}{p-1}} - R^{\frac{p-N}{p-1}} \)^{\frac{p-1}{p-N}} \)^{(N-1)(\frac{p}{p-1})} s^{\frac{p(1-N)}{p-1}} s^{N-1} ds \notag \\
	&= \int_0^R \frac{F(V(s))}{s^{\frac{p(N-1)}{p-1}} \( s^{\frac{p-N}{p-1}} - R^{\frac{p-N}{p-1}} \)^{\frac{p(N-1)}{N-p}}} s^{N-1} ds \notag \\
	&= \int_0^R \frac{F(V(s))}{\( 1 - \(\frac{s}{R}\)^{\frac{N-p}{p-1}} \)^{\frac{p(N-1)}{N-p}}} s^{N-1} ds.
\end{align}
Thus again Proposition \ref{Prop:Polar} yields that
\begin{align*}
	\int_{\re^N} F(u(x)) dx &= \omega_{N-1} \int_0^\infty F(U(r)) r^{N-1} dr \\
	&\stackrel{\eqref{E3}}{=} \frac{\omega_{N-1}}{N \kappa_N} N \kappa_N \int_0^R  \frac{F(V(s))}{\( 1 - \(\frac{s}{R}\)^{\frac{N-p}{p-1}} \)^{\frac{p(N-1)}{N-p}}} s^{N-1} ds \\
	&\stackrel{\eqref{Polar}}{=} \frac{\omega_{N-1}}{N \kappa_N} \int_{\mathcal{W}_R} \frac{F(v(y))}{\( 1 - \(\frac{H^0(y)}{R}\)^{\frac{N-p}{p-1}} \)^{\frac{p(N-1)}{N-p}}} dy.
\end{align*}
\end{proof}

Next, let us replace the roles of $u$ and $v$ in Proposition \ref{Prop:Ioku}.
That is, let $v = v(y)$ be a radially symmetric function on $B_R = \{ x \in \re^N \ : \ |x| < R \}$, here $R > 0$. 
Thus, there exists a function $V$ defined on $[0, R)$ such that $v(y) = V(|y|)$.
Also let $u = u(x)$ be a Finsler radially symmetric function on $\re^N$ of the form $u(x) = U(H^0(x))$, $U = U(r)$, $r \in [0, +\infty)$. 
Assume $u$ and $v$ are related by the transformation
\begin{equation}
\label{Ioku2}
\begin{cases}
	&r = H^0(x), x \in \re^N, \\
	&s = |y|, y \in B_R \subset \re^N, \\
	&r^{\frac{p-N}{p-1}} = s^{\frac{p-N}{p-1}} - R^{\frac{p-N}{p-1}}, \\
	&u(x) = U(r) = V(s) = v(y).
\end{cases}
\end{equation}
Then as before, under the transformation \eqref{Ioku2}, we have the following equivalence.

\begin{proposition}
\label{Prop:Ioku2}
Let $u, v$ be as above. Let $F: \re \to \re$ be continuous.
Then we have
\begin{align*}
	&\int_{B_R} |\nabla v|^p dy = \frac{\omega_{N-1}}{N \kappa_N} \int_{\re^N} H(\nabla u)^p dx, \\
	&\int_{B_R} F(v(y)) dy = \frac{\omega_{N-1}}{N \kappa_N} \int_{\re^N} \frac{F(u(x))}{\( 1 + \(\frac{H^0(x)}{R}\)^{\frac{N-p}{p-1}} \)^{\frac{p(N-1)}{N-p}}} dx.
\end{align*}
\end{proposition}

The proof is similar as above, and we omit it here.

\section{Functional inequalities for symmetric functions}
\label{section:inequalities}

In this section, we will prove several functional inequalities which hold for functions in the appropriate Sobolev space with some symmetry.
Though many similar inequalities can be derived by the same idea, we record here few of them.

Following inequalities are direct consequences of Proposition \ref{Prop:Ioku}, Proposition \ref{Prop:Ioku2},
and the known inequalities on $\re^N$ or $B_R$, with the information of extremals.

\subsection{The sharp $L^p$-Sobolev inequality}

First we treat the sharp $L^p$-Sobolev inequality.

\begin{theorem}
\label{Theorem:Sobolev}
Let $N \ge 2$, $1 < p < N$ and $p^* = \frac{Np}{N-p}$.
Then for any Finsler radially symmetric function $v \in W^{1,p}_0(\W_R)$, the inequality
\[
	\tilde{S}_{N,p} \( \int_{\W_R} \frac{|v(y)|^{p^*}}{\( 1 - (\frac{H^0(y)}{R})^{\frac{N-p}{p-1}} \)^{\frac{p(N-1)}{N-p}}} dy \)^{p/p^*} \le \int_{\W_R} H(\nabla v(y))^p dy
\]
holds true.
Here 
\[
	\tilde{S}_{N,p} = S_{N,p} \(\frac{\omega_{N-1}}{N \kappa_N} \)^{p/p^* - 1}
\]
and $S_{N,p}$ is defined in \eqref{SNp}.
The equality occurs if and only if $v$ is of the form
\[
	v(y) = \( a + b \( (H^0(y))^{\frac{p-N}{p-1}} - R^{\frac{p-N}{p-1}} \)^{\frac{p}{p-N}} \)^{\frac{p-N}{p}} 
\]
for some $a, b >0$.
\end{theorem}

\begin{proof}
For any Finsler radially symmetric function $v \in W^{1,p}_0(\W_R)$, define $u \in W^{1,p}(\re^N)$ as $u(x) = U(r) = V(s) = v(y)$, where
$r = |x|$, $x \in \re^N$ and $s = H^0(y)$, $y \in \W_R$.
$U$ and $V$ are defining functions of $u, v$ respectively.
Then the $L^p$-Sobolev inequality \eqref{Sobolev} for $u$ (with the information of extremals) and Proposition \ref{Ioku} yield the result.
\end{proof}

\subsection{Gagliardo-Nirenberg inequalities}

The following optimal Gagliardo-Nirenberg inequality was established by del Pino and Dolbeault \cite{DD(JMPA)}, which includes the sharp $L^2$-Sobolev inequality as a special case.
Let $q > 1$ if $N = 2$ and $1 < q \le \frac{N}{N-2}$ if $N \ge 3$. Then for any function $u \in L^{q+1}(\re^N) \cap L^{2q}(\re^N)$ with $\nabla u \in L^2(\re^N)$, 
the inequality
\begin{equation}
\label{GN}
	\| u \|_{L^{2q}(\re^N)} \le A \| \nabla u \|_{L^2(\re^N)}^{\theta} \| u \|_{L^{q+1}(\re^N)}^{1-\theta}
\end{equation}
holds true where 
\[
	\frac{1}{2q} = \theta \( \frac{1}{2} - \frac{1}{N} \) + (1-\theta) \frac{1}{q+1},
\]
i.e., 
\begin{equation}
\label{GN_theta}
	\theta = \frac{N(q-1)}{q(N+2-(N-2)q)},
\end{equation}
and
\begin{equation}
\label{GN_best}
	A = \( \frac{(q-1)(q+1)}{2\pi N} \)^{\theta/2} \( \frac{2(q+1) - N(q-1)}{2(q+1)} \)^{1/(2q)} \( \frac{\Gamma(\frac{q+1}{q-1})}{\Gamma(\frac{q+1}{q-1}-\frac{N}{2})} \)^{\theta/N}.
\end{equation}
$A$ is optimal and the equality in \eqref{GN} holds if and only if $u$ is a constant multiple and translation of functions
\begin{equation}
\label{GN_extremal}
	\phi_{\sigma}(x) = \(\frac{1}{\sigma^2 + |x|^2}\)^{\frac{1}{q-1}}, \quad \sigma > 0.
\end{equation}
Note that the sharp $L^2$-Sobolev inequality is recovered when $q = \frac{N}{N-2}$.

By the same argument in proving Theorem \ref{Theorem:Sobolev}, but in this case we use the transformation 
\[
\begin{cases}
	&r = |x|, x \in \re^N, \\
	&s = H^0(y), y \in \mathcal{W}_R \subset \re^N, \\
	&r^{2-N} = s^{2-N} - R^{2-N}, \\
	&u(x) = U(r) = V(s) = v(y),
\end{cases}
\]
we obtain the following:

\begin{theorem}
\label{Theorem:GN}
Let $N \ge 3$ and $1 < q \le \frac{N}{N-2}$.
Then for any Finsler radially symmetric function $v \in W^{1,2}_0(\W_R)$, the inequality
\begin{align*}
	\( \int_{\W_R} \frac{|v(y)|^{2q}}{\( 1 - (\frac{H^0(y)}{R})^{N-2} \)^{\frac{2(N-1)}{N-2}}} dy \)^{\frac{1}{2q}} 
	&\le \tilde{A} \( \int_{\W_R} H(\nabla v(y))^2 dy \)^{\frac{\theta}{2}} \\ 
	\times \( \int_{\W_R} \frac{|v(y)|^{q+1}}{\( 1 - (\frac{H^0(y)}{R})^{N-2} \)^{\frac{2(N-1)}{N-2}}} dy \)^{\frac{1-\theta}{q+1}}
\end{align*}
holds true.
Here $\tilde{A} = \( \frac{\omega_{N-1}}{N \kappa_N} \)^{\frac{\theta}{N}} A$, where $\theta$ is in \eqref{GN_theta} and $A$ is the constant in \eqref{GN_best}.
The equality occurs if and only if $v$ is of the form
\[
	v(y) = C \( \frac{1}{\sigma^2 + \( (H^0(y))^{2-N} - R^{2-N} \)^{\frac{2}{2-N}} } \)^{\frac{1}{q-1}}, \quad \sigma > 0.
\]
for some $\sigma >0$ and $C \ne 0$.
\end{theorem}

\vspace{1em}
Nash's inequality 
\begin{equation}
\label{Nash}
	\| u \|^{1 + 2/N}_{L^2(\re^N)} \le B \| \nabla u \|_{L^2(\re^N)} \| u \|_{L^1(\re^N)}^{2/N}
\end{equation}
which holds true where for any $u \in W^{1,2}(\re^N) \cap L^1(\re^N)$, is also a special case of the Gagliardo-Nirenberg inequality.
Let
\[
	\la^N_1(B) = \inf \left\{ \frac{\int_B |\nabla u|^2 dx}{\int_B |u|^2 dx} \ | \ u \in W^{1,2}(B), \int_B u dx = 0, \text{$u$ is radially symmetric} \right \}
\]
denote the first non-zero Neumann eigenvalue of $-\Delta$ for all radially symmetric functions $u \in W^{1,2}(B)$ on the unit ball $B \subset \re^N$ with the average zero.
Carlen and Loss \cite{Carlen-Loss} proved that the best constant in the right-hand side is
\begin{equation}
\label{Nash_best}
	B = 2 \( 1 + \frac{N}{2} \)^{1 + \frac{2}{N}} N^{-1+\frac{2}{N}} \( \frac{1}{\la^N_1(B) \omega_{N-1}^{2/N}} \)
\end{equation}
and that the equality in \eqref{Nash} holds if and only if $u$ is a constant multiple, scaling, and translation of the function
\begin{equation}
\label{Nash_extremal}
	\Psi(|x|) = \begin{cases}
	U(|x|) - U(1), &\quad (|x| \le 1) \\ 
	0, &\quad (|x| \ge 1),
	\end{cases}
\end{equation}
where
\[
	U(r) = r^{\frac{2-N}{2}} J_{\frac{N-2}{2}}( \mu r),
\]
here $J_{\nu}$ denotes the Bessel function of order $\nu$, and $\mu$ is the first positive zero of $J_{\frac{N}{2}}$, which turns out to be the first positive zero of $U'(r)$.
By this notation, $\la_1^N(B)$ can be written as $\la_1^N(B) = \mu^2$.

Then as before, we have the next theorem.

\begin{theorem}
\label{Theorem:Nash}
Let $N \ge 3$.
Then for any Finsler radially symmetric function $v \in W^{1,2}_0(\W_R)$, the inequality
\begin{align*}
	\( \int_{\W_R} \frac{|v(y)|^2}{\( 1 - (\frac{H^0(y)}{R})^{N-2} \)^{\frac{2(N-1)}{N-2}}} dy \)^{1 + \frac{2}{N}} 
	&\le \tilde{B} \( \int_{\W_R} H(\nabla v(y))^2 dy \) \\
	\times \( \int_{\W_R} \frac{|v(y)|}{\( 1 - (\frac{H^0(y)}{R})^{N-2} \)^{\frac{2(N-1)}{N-2}}} dy \)^{\frac{4}{N}}
\end{align*}
holds true.
Here $\tilde{B} = \( \frac{\omega_{N-1}}{N \kappa_N} \)^{\frac{2}{N}} B$, where $B$ is the constant in \eqref{Nash_best}.
The equality occurs if and only if $v$ is of the form
\[
	v(y) = C \Psi\(\la \( (H^0(y))^{2-N} - R^{2-N} \)^{\frac{1}{2-N}} \)
\]
for some $\la > 0$ and $C \ne 0$.
\end{theorem}

\subsection{The Euclidean $L^p$-logarithmic Sobolev inequality}

The Euclidean $L^p$-logarithmic Sobolev inequality states that
\begin{equation}
\label{ELSp}
	\int_{\re^N} |u|^p \log |u|^p dx \le \frac{N}{p} \log \( \mathcal{L}_p \int_{\re^N} |\nabla u|^p dx \)
\end{equation}
holds for any function $u \in W^{1,p}(\re^N)$ such that $\int_{\re^N} |f|^p dx = 1$. 
This form of inequality was first proved by Weissler \cite{Weissler} for $p = 2$,  Ledoux \cite{Ledoux} for $p = 1$, del Pino and Dolbeault \cite{DD(JFA)} for $1 \le p < N$,
and finally generalized by Gentil \cite{Gentil} for $1 \le p < \infty$.
Actually, Gentil extends the result in \cite{DD(JFA)} not only for all $p \ge 1$, but also for any norm on $\re^N$ other than usual Euclidean norm.
Here the sharp constant $\mathcal{L}_p$ is given by
\begin{equation}
\label{L_p}
	\begin{cases}
	&\mathcal{L}_1 = \frac{1}{N} \pi^{-1/2} \(\Gamma(N/2+1)\)^{1/N}, \quad p = 1 \\
	&\mathcal{L}_{p} = \frac{p}{N} \(\frac{p-1}{e}\)^{p-1} \pi^{-p/2} \(\frac{\Gamma(N/2+1)}{\Gamma(N/p' + 1)} \)^{p/N}, \quad p > 1.
	\end{cases}
\end{equation}
where $p' = \frac{p}{p-1}$ for $p > 1$.
For $p=1$, Beckner \cite{Beckner} proved that the extremal functions for \eqref{ELSp} are the characteristic functions of balls. 
For $1 < p < N$, it is proved in \cite{DD(JFA)} that the extremal functions of \eqref{ELSp} must be of the form
\begin{align*}
	u(x) = U(|x|) = C(N, p) \exp \( -\frac{1}{\sigma} |x|^{p'} \)
\end{align*}
where $\sigma > 0$ and
\begin{equation}
\label{CNp}
	C(N,p) = \( \pi^{N/2} \(\frac{\sigma}{p}\)^{N/p'} \frac{\Gamma(N/p'+1)}{\Gamma(N/2+1)} \)^{-1/p},
\end{equation}
and its translation.

From this and the same argument as in Theorem \ref{Theorem:Sobolev}, we have

\begin{theorem}
\label{Theorem:log-Sobolev}
Let $1 \le p < N$ and $R > 0$. 
Then 
\[
	\( \frac{\omega_{N-1}}{N \kappa_N} \) \int_{\W_R} \frac{|v(y)|^p \log |v(y)|^p}{\( 1 - \(\frac{H^0(y)}{R}\)^{\frac{N-p}{p-1}} \)^{\frac{p(N-1)}{N-p}}} dy 
\le \frac{N}{p} \log \( \tilde{\mathcal{L}}_p \int_{\W_R} H(\nabla v)^p dy \)
\]
holds true for any Finsler radially symmetric function $v$ satisfying
\[
	\( \frac{\omega_{N-1}}{N \kappa_N} \) \int_{\W_R} \frac{|v(y)|^p}{\( 1 - \(\frac{H^0(y)}{R}\)^{\frac{N-p}{p-1}} \)^{\frac{p(N-1)}{N-p}}} dy = 1.
\]
Here $\tilde{\mathcal{L}}_p = \( \frac{\omega_{N-1}}{N \kappa_N} \) \mathcal{L}_p$.
The equality holds if and only if $v$ is of the form
\begin{align*}
	v(y) = C(N,p) \exp \( -\frac{1}{\sigma} \( (H^0(y))^{\frac{p-N}{p-1}} - R^{\frac{p-N}{p-1}} \)^{\frac{-p}{N-p}} \)
\end{align*}
for $1 < p < N$, where $C(N,p)$ is defined in \eqref{CNp}.
When $p=1$, then the extremals are the characteristic functions of Wulff balls.
\end{theorem}

\subsection{The Poincar\'e inequality on balls}

The $L^p$-Poincar\'e inequality on balls states that
\begin{equation}
\label{Poincare}
	\la_1(B_1) \int_{B_R} |v(y)|^p dy \le R^p \int_{B_R} |\nabla v(y)|^p dy
\end{equation}
holds for any function $v \in W^{1,p}_0(B_R)$, where $B_R \subset \re^N$ is a ball with radius $R >0$. 
Here, $\la_1(B_1)$ is the first eigenvalue of $-\Delta_p$ ($p$-Laplacian) with the Dirichlet boundary condition on the unit ball $B_1 \subset \re^N$.
To the authors' knowledge, the explicit expression is not known for $\la_1(B_1)$ unless $p = 2$.
The equality in \eqref{Poincare} holds if and only if $v$ is a constant multiple of the first eigenfunction of $-\Delta_p$ on $B_R$, which we denote $\phi_R \in W^{1,p}_0(B_R)$.
Known regularity and symmetry results assure that the first eigenfunction of $-\Delta_p$ is $C^{1,\alpha}$ for some $\alpha \in (0,1)$ and radially symmetric. 
Thus we can write $\phi_R(y) = \Phi_R(|y|)$, $y \in B_R$, for some $C^1$-function $\Phi_R$ on $[0, R)$ with $\Phi_R(R) = 0$.

By these facts and Proposition \ref{Prop:Ioku2}, we have the following.

\begin{theorem}
\label{Theorem:Poincare}
Let $1 \le p < N$ and $R > 0$ is arbitrarily given. 
Then 
\[
	\la_1(B_1) \int_{\re^N} \frac{|u(x)|^p}{\( 1 + \(\frac{H^0(x)}{R}\)^{\frac{N-p}{p-1}} \)^{\frac{p(N-1)}{N-p}}} dx \le R^p \int_{\re^N} H(\nabla u)^p dx
\]
holds true for any Finsler radially symmetric function $u(x) = U(H^0(x)) \in W^{1,p}(\re^N)$ such that $u(\infty) = 0$.
For fixed $R > 0$, the equality holds if and only if $u$ is the constant multiple of
\[
	\Phi_R \( \( (H^0(x))^{\frac{p-N}{p-1}} + R^{\frac{p-N}{p-1}} \)^{\frac{p-1}{p-N}} \)
\]
where $\Phi_R \in C^1([0,R))$ is such that $\phi_R(y) = \Phi_R(|y|)$, $y \in B_R$, is the first eigenfunction of $-\Delta_p$ on $B_R$.
\end{theorem}

\begin{proof}
The proof follows for a given Finsler radially symmetric function $u \in W^{1,p}(\re^N)$, define a radially symmetric function $v(y) = V(|y|)$ as 
$v(y) = V(s) = U(r) = u(x)$ where $s = |y|$, $y \in B_R$ and $r = H^0(x)$, $x \in \re^N$.
Then use Proposition \ref{Ioku2}.
\end{proof}

\subsection{The sharp $L^p$-Sobolev trace inequality}

Let $N \ge 3$, $1 < p < N$ and put $p_* = \frac{(N-1)p}{N-p}$.
Let
\[
	\re^N_{+} = \{ (x, t) \in \re^{N-1} \times \re_{+} \}
\]
denote the upper half space and identify $\pd \re^N_{+} = \{ (x, 0) \ | \ x \in \re^{N-1} \} \simeq \re^{N-1}$.
The $L^p$-sharp Sobolev trace inequality
\begin{equation}
\label{Sobolev_trace}
	S_{T,N,p} \( \int_{\pd \re^N} |u(x,0)|^{p^*} dx \)^{p/p_*} \le \int_{\re^N_{+}} |\nabla u(x,t)|^p dx dt
\end{equation}
holds for every $u \in \dot{W}^{1,p}(\re^N) = \{ u \in L^{p^*}(\re^N) \ | \ \nabla u \in L^p(\re^N) \}$, where $p^* = \frac{Np}{N-p}$.
The best constant $\tilde{S}_{N,p}$ is 
\begin{equation}
\label{SNp_trace}
	S_{T,N,p} = \pi^{(p-1)/2} \(\frac{N-p}{p-1}\)^{p-1} \( \frac{\Gamma(\tfrac{N-1}{2(p-1)})}{\Gamma(\tfrac{(N-1)p}{2(p-1)})} \)^{\tfrac{p-1}{N-1}}
\end{equation}
see Escobar \cite{Escobar} for $p = 2$ and Nazaret \cite{Nazaret} for $p \in (1,N)$. 
Also it is proven that the best constant $S_{T,N,p}$ is achieved by a family of functions
\begin{equation}
\label{extremal_trace}
	\phi_{\eps}(x,t) = \eps^{-\frac{N-p}{p}} \phi\( \frac{x}{\eps}, \frac{t}{\eps} \) = \( \frac{\eps^{\frac{2}{p}}}{(\eps + t)^2 + |x|^2} \)^{\frac{N-p}{2(p-1)}} \quad \eps > 0
\end{equation}
where 
\[
	\phi(x, t) = \( \frac{1}{(1 + t)^2 + |x|^2} \)^{\frac{N-p}{2(p-1)}},
\]
and its constant multiple and the translation $C \phi_{\eps}(x - x_0, t)$ for $C \ne 0$ and  $x_0 \in \re^{N-1}$.

Let $\mathcal{W}^{N-1}_R$ denote the Wulff ball of radius $R > 0$ in $\re^{N-1}$:
\[
	\mathcal{W}^{N-1}_R = \{ y \in \re^{N-1} \ | \ H^0(y) < R \},
\]
and let $1 < p < N-1$, $N \ge 3$.
We relate functions $u = u(x, t)$ on $\re^N_{+}$ of the form $u(x, t) = U(|x|, t)$, and functions on $v = v(y,t)$ on $\mathcal{W}_R^{N-1} \times \re_{+}$
of the form $v(y, t) = V(H^0(y), t)$ for some $U$ and $V$  by the relation
\begin{equation}
\label{Ioku3}
\begin{cases}
	&r = |x|, x \in \re^{N-1}, \\
	&s = H^0(y), y \in \mathcal{W}^{N-1}_R \subset \re^N, \\
	&r^{\frac{p-(N-1)}{p-1}} = s^{\frac{p-(N-1)}{p-1}} - R^{\frac{p-(N-1)}{p-1}}, \\
	&u(x, t) = U(r, t) = V(s, t) = v(y, t).
\end{cases}
\end{equation}
Under the transformation \eqref{Ioku3}, we have the following equivalence.

\begin{proposition}
\label{Prop:Ioku3}
Let $u, v$ be as above. 
Put
\begin{equation}
\label{A_R}
	A_R(y) = \left\{ 1 - \(\frac{H^0(y)}{R}\)^{\frac{N-1-p}{p-1}} \right\}^{\frac{2(N-1)}{N-1-p}}, \quad y \in \mathcal{W}^{N-1}_R.
\end{equation}
Then we have
\begin{align*}
	&\int_0^{\infty} \int_{\re^{N-1}} |\nabla_{x,t} u(x,t)|^p dxdt \\
	&= \frac{\omega_{N-2}}{(N-1) \kappa_{N-1}} \int_0^{\infty} \int_{\mathcal{W}^{N-1}_R} \( H(\nabla v)^2 A_R(y) + \(\frac{\pd v}{\pd t}\)^2 \)^{p/2} A_R(y)^{-p/2} dy.
\end{align*}
\end{proposition}

By Proposition \ref{Prop:Ioku3} and Proposition \ref{Prop:Ioku} ($N$ replaced by $N-1$), we have the following Sobolev trace inequality involving the Finsler norm:

\begin{theorem}
\label{Theorem:Sobolev_trace}
Let $N \ge 3$, $1 < p < N-1$ and $p_* = \frac{(N-1)p}{N-p}$.
For all functions $v \in W^{1,p}(\W^{N-1}_R \times \re_{+})$ of the form $v(y, t) = V(H^0(y), t)$ for a function $V = V(s,t)$, $(s, t) \in [0, R) \times \re_{+}$,
the inequality
\begin{align*}
	&\tilde{S}_{T,N,p} \( \int_{\W^{N-1}_R} |v(y, 0)|^{p_*} A_R(y)^{-p/2} dy \)^{p/p_*} \\
	&\le \int_0^{\infty} \int_{\W^{N-1}_R} \( H(\nabla v)^2 A_R(y) + \(\frac{\pd v}{\pd t}\)^2 \)^{p/2} A_R(y)^{-p/2} dy
\end{align*}
holds true, where $A_r(y)$ is defined in \eqref{A_R} and
\[
	\tilde{S}_{T,N,p} = S_{T,N,p} \( \frac{\omega_{N-2}}{(N-1)\kappa_{N-1}} \)^{\frac{p-p_*}{p_*}},
\]
here $S_{T,N,p}$ is the same constant in \eqref{SNp_trace}. 
The equality occurs if $v$ is of the form
\[
	v(y, t) = C \( \frac{\eps^{\frac{2}{p}}}{(\eps + t)^2 + \( (H^0(y))^{\frac{p-(N-1)}{p-1}} - R^{\frac{p-(N-1)}{p-1}} \)^{\frac{2(p-1)}{p-(N-1)}}} \)^{\frac{N-p}{2(p-1)}}
\]
for some $\eps>0$ and $C \ne 0$.
\end{theorem}

\begin{remark}
As in the classical case \eqref{Sobolev_trace}, 
it is not known that the functions $v(y,t)$ in Theorem \ref{Theorem:Sobolev_trace} is the only extremizers or not for the best constant $\tilde{S}_{T,N,p}$, 
except for $p=2$.
\end{remark}

\begin{proof}
Proof of the theorem is as before: For such a function $v(y, t) = V(H^0(y), t)$, define a new function $u(x,t)$ by the relation \eqref{Ioku3},
and use the sharp $L^p$ Sobolev trace inequality for $u$. The information of the extremals comes from the transformation and \eqref{extremal_trace}.
\end{proof}

\subsection{Trudinger-Moser inequality}

We can consider the same type of transformation between $u$ and $v$ on the {\it different dimension} of spaces.
For example, let $u = u(x)$ be a Finsler radially symmetric function on $\re^N$, $N \ge 3$,
and let $v = v(y)$ be a radially symmetric function on $B_R \subset \re^2$ for some $R > 0$.
Let us assume that $u$ and $v$ are related with each other by the transformation
\begin{equation}
\label{Ioku4}
\begin{cases}
	&r = H^0(x), x \in \re^N, \\
	&s = |y|, y \in B_R \subset \re^2, \\
	&r^{2-N} = \log \frac{R}{s}, \\
	&u(x) = U(r) = V(s) = v(y).
\end{cases}
\end{equation}
Under the transformation \eqref{Ioku4}, we have 

\begin{proposition}
\label{Prop:Ioku4}
Let $u, v$ be as above. Let $F: \re \to \re$ be continuous.
Then we have
\begin{align*}
	&\int_{B_R} |\nabla v(y)|^2 dy = \frac{2\pi}{(N-2)N \kappa_N} \int_{\re^N} H(\nabla u)^2 dx, \\
	&\int_{B_R} F(v(y)) dy = \int_{\re^N} F(u(x)) W_{H,R}(x) dx,
\end{align*}
where
\begin{equation}
\label{W_HR}
	W_{H,R}(x) = \( \frac{2\pi(N-2)}{N \kappa_N} \) \frac{R^2}{(H^0(x))^{2(N-1)} \exp \{ 2 (H^0(x))^{2-N} \}}.
\end{equation}
\end{proposition}

\begin{proof}
By \eqref{Ioku4}, we see that if $r$ runs from $0$ to $\infty$, then $s$ runs from $0$ to $R$, and vice versa.
Also by differentiating the relation,
we have
\[
	\frac{dr}{ds} = \frac{r^{N-1}}{N-2} \frac{1}{s}.
\]
Since $V'(s) = U'(r) (\frac{dr}{ds})$, we compute
\begin{align*}
	\int_{B_R} |\nabla v(y)|^2 dy &= 2\pi \int_0^R |V'(s)|^2 s ds \\
	&= 2\pi \int_0^{\infty} |U'(r)|^2 \(\frac{dr}{ds}\)^2 s(r) \(\frac{ds}{dr}\) dr \\
	&= 2\pi \int_0^{\infty} |U'(r)|^2 \frac{r^{N-1}}{s(r)(N-2)} s(r) dr \\
	&= \frac{2\pi}{(N-2)N \kappa_N} N \kappa_N \int_0^{\infty} |U'(r)|^2 r^{N-1} dr \\
	&= \frac{2\pi}{(N-2)N \kappa_N} \int_{\re^N} |U'(H^0(x))|^2 dx \\
	&= \frac{2\pi}{(N-2)N \kappa_N} \int_{\re^N} H(\nabla u)^2 dx,
\end{align*}
here we used
\[
	H\(\nabla u(x)\) = H\(U'(H^0(x)) \nabla H^0(x)\) = |U'(H^0(x)|
\]
in the last equality.
Also, we compute
\begin{align*}
	\int_{B_R} F(v(y)) dy &= 2\pi \int_0^R F(V(s)) s ds \\
	&= 2\pi \int_0^{\infty} F(U(r)) s(r) \(\frac{ds}{dr}\) dr \\
	&= 2\pi \int_0^{\infty} F(U(r)) s(r) \frac{(N-2) s(r)}{r^{N-1}} dr \\
	&= \frac{2\pi(N-2)}{N \kappa_N} N \kappa_N \int_0^{\infty} F(U(r)) \frac{s(r)^2}{r^{2(N-1)}} r^{N-1} dr \\
	&= \frac{2\pi(N-2)}{N \kappa_N} \int_{\re^N} F(u(x)) \frac{R^2 \exp \{ -2 (H^0(x))^{2-N} \}}{(H^0(x))^{2(N-1)}} dx \\
	&= \int_{\re^N} F(u(x)) W_{H,R}(x) dx.
\end{align*}
\end{proof}

Trudinger-Moser inequality \cite{Moser}, \cite{Trudinger} on $B_R \subset \re^2$ states that
\begin{equation}
\label{TM}
	\sup \left\{ \int_{B_R} e^{4\pi v(y)^2} dy \ | \ v \in W^{1,2}_0(B_R), \int_{B_R} |\nabla v|^2 dy \le 1 \right\} < \infty
\end{equation}
and the supremum is attained by a radially symmetric function 
\begin{equation}
\label{TM_extremal}
	v(y) = V_*(|y|), \quad y \in B_R \subset \re^2,
\end{equation}
see \cite{Carleson-Chang}.

By this fact and Proposition \ref{Prop:Ioku4}, we have the following.

\begin{theorem}
\label{Theorem:TM}
Let $N \ge 3$ and put
\[
	\mathcal{B} = \left\{ u \in W^{1,2}_{0, FR}(\re^N) \ | \ \int_{\re^N} H(\nabla u)^2 dx \le \frac{N(N-2)\kappa_N}{2\pi} \right\},
\]
where $W^{1,2}_{0, FR}(\re^N)$ denotes the set of Finsler radially symmetric functions in $W^{1,2}_0(\re^N)$. 
Then we have
\begin{align*}
	\sup_{u \in \mathcal{B}} \int_{\re^N} e^{4\pi u(x)^2} W_{H,R}(x) dx < \infty
\end{align*}
where $W_{H,R}(x)$ is defined in \eqref{W_HR}.
Moreover, the supremum is attained by
\[
	u(x) = V_*(R e^{-(H^0(x))^{2-N}}), \quad x \in \re^N
\]
where $V_*$ is a function in \eqref{TM_extremal}.
\end{theorem}

%
%
\vspace{1em}\noindent
{\bf Declarations}

\vspace{1em}

{\small
\vspace{1em}\noindent
{\bf Funding}

\noindent
Not applicable.

\vspace{1em}\noindent
{\bf Availability of data and materials}

\noindent
Not applicable.

\vspace{1em}\noindent
{\bf Competing interests}

\noindent
The authors declare that they have no competing interests.

\vspace{1em}\noindent
{\bf Author's contribution}

\noindent
All authors have taken part in this research equally and they read and approved the final manuscript.

\vspace{1em}\noindent
{\bf Acknowledgments}

The authors thank Prof. Megumi Sano and Prof. Norisuke Ioku for fruitful discussions and giving us comments on this topic.
This work was partly supported by Osaka City University Advanced Mathematical Institute MEXT Joint Usage / Research Center on Mathematics and Theoretical Physics JPMXP0619217849.
The second author (F.T.) was supported by JSPS KAKENHI Grant-in-Aid for Scientific Research (B), JP19H01800, and JSPS Grant-in-Aid for Scientific Research (S), JP19H05597.
}


\end{document}